\documentclass[a4papersize]{article}
\usepackage{amsmath,amsthm, amsxtra,amssymb,latexsym, amscd}
\usepackage{color}
\usepackage[mathscr]{eucal}
\newtheorem{theorem}{Theorem}[section]

\newtheorem{lemma}[theorem]{Lemma}

\newtheorem{example}[theorem]{Example}
\newtheorem{question}[theorem]{Question}

\newtheorem{remark}[theorem]{Remark}

\DeclareMathOperator{\depth}{depth}

\DeclareMathOperator{\Ker}{Ker}

\DeclareMathOperator{\Ann}{Ann}

\DeclareMathOperator{\Soc}{Soc}
\DeclareMathOperator{\im}{Im}

\DeclareMathOperator{\ir}{ir}
\begin{document}
\large

\centerline{\Large {\bf BUCHSBAUM MODULES AND  PARAMETER IDEALS}}
\smallskip
\centerline{\Large {\bf  UNDER FLAT EXTENSIONS}}

\medskip

\vskip 0.7cm
\centerline {NGUYEN THI HONG LOAN}
\centerline{ Vinh University, Nghe An, Vietnam}
\centerline{ E-mail: loannth@vinhuni.edu.vn}
\vskip 0.3cm

\centerline {LE THANH NHAN}
\centerline {Institute of Mathematics}
\centerline{Vietnam Academy of Science and Technology, Hanoi, Vietnam}
\centerline {E-mail: lttnhan@math.ac.vn}

\vskip 0.3cm
\centerline {PHAM HUNG QUY}
\centerline{Department of Mathematics,  FPT University}
\centerline{Hoa Lac Hi-tech Park, Hanoi, Vietnam}
\centerline{E-mail:quyph@fe.edu.vn}

\vskip 0.7cm

\centerline{Dedicated to our advisor, Professor Nguyen Tu Cuong, on the occasion of his 75th birthday.}
\vskip 0.7cm

 \noindent{\bf Abstract}  {\footnote{ {\it{Key words and phrases: }} Flat extension; Buchsbaum module; parameter ideal; local cohomology module. \hfill\break {\it{2020 Subject  Classification: }} 13H10, 13D45, 13B40.\hfill\break	This research is funded by Vietnam 	National Foundation for Science and Technology Development (NAFOSTED).} Let $(R,\frak m)$ and $(S, \frak n)$ be commutative Noetherian local rings, let $\varphi: R\to S$ be a flat local homomorphism.  In this paper, we first characterize the ascent and descent of Buchsbaum modules and generalized Cohen-Macaulay modules via the flat extension  $\varphi$. Then, we provide relationships between the parameter ideals of a finitely generated $R$-module $M$ and those of $M\otimes_RS$ under the assumption that $\frak n=\frak mS.$ 

\section{Introduction}  

Throughout this paper, let $(R,\frak m)$ and $(S, \frak n)$ denote commutative Noetherian local rings. Let  $\varphi: R\rightarrow S$  be a flat local homomorphism and $M$  a finitely generated $R$-module.

Let $\dim_R(M)=d.$ For a system of parameters (s.o.p for short) $x_1, \ldots , x_d$    of $R$,  we set  
$$I(x_1, \ldots , x_d; M)=\ell_R(M/(x_1, \ldots , x_d)M)-e(x_1, \ldots , x_d; M).$$ If $I(x_1, \ldots , x_d; M)$ is a constant not depending on the choice of s.o.p $x_1, \ldots , x_d$, then $M$ is said to be a {\it Buchsbaum module}. The formal definition and systematic theory of Buchsbaum modules were introduced and developed by J. St\"{u}ckrad and W. Vogel in the 1970s, see \cite{SV}.  Set $I(M)=\sup I(x_1, \ldots , x_d; M)$, where the supremum runs over all s.o.p $x_1, \ldots , x_d$ of $M$. Following N. T. Cuong, P. Schenzel and N. V. Trung \cite{CST}, if $I(M)<\infty$ then $M$ is said to be a {\it generalized Cohen-Macaulay module}.  Note that $M$ is generalized Cohen-Macaulay if and only if $H^i_{\frak m}(M)$ is of finite length for all $i<d,$ if and only if $M$ admits a standard s.o.p $x_1, \ldots , x_d$, i.e $I(x_1, \ldots , x_d; M)=I(x_1^2, \ldots , x_d^2; M)$ (see \cite{CST, Sch, Tr}).  Buchsbaum modules and generalized Cohen-Macaulay modules were extensively studied from many different aspects, and they are familiar research objects in commutative algebra and algebraic geometry nowadays.

Some classical results in commutative algebra, initiated by A. Grothendieck \cite[6.3.5]{G}, provided many important properties that can be transferred between $R$ and $S$ in the relation to the closed fiber $S/\frak mS$, a phenomenon known as ascent and descent under the flat extension $\varphi$. We can find the following ascent and descent properties in the references \cite{Av, G, Hart, Mat, WITO}): $S$ is a complete intersection (resp., Gorenstein, Cohen-Macaulay) ring if and only if so are $R$ and $S/\frak mS$. For regular local rings, we have a weaker form of the same result: if $R$ and $S/\frak mS$ are regular then so is $S$; if $S$ is regular then  $R$ is regular, but $S/\frak mS$ need not be regular. Also, the ascent and descent under a flat extension for certain classes of modules, such as Noetherian, Artinian, and injective modules, have been investigated in \cite{ChN, Fox, FWW}.

It is well-known that $M\otimes_RS$ is   Cohen-Macaulay if and only if $M$  and $S/\frak mS$ are Cohen-Macaulay. The first purpose of this paper is to characterize the ascent and descent via the flat extension $\varphi: R\to S$ for the classes of Buchsbaum modules and generalized Cohen-Macaulay modules. The following theorem is the first main result of this paper.

\begin{theorem} \label{T:1a} Suppose that $M$ is not Cohen-Macaulay. Then 
\begin{itemize}
\item[{\rm (a)}] $M\otimes_RS$ is  generalized Cohen-Macaulay if and only if  $M$ is   generalized Cohen-Macaulay and $\dim(S/\frak mS)=0$. In this case we have $I(M\otimes_RS)=\ell_S(S/\frak mS)~I(M).$
\item[{\rm (b)}] $M\otimes_RS$ is  Buchsbaum  if and only if $M$ is  Buchsbaum and $\frak n=\frak mS.$ In this case we have $I(M\otimes_RS)=I(M).$		
	\end{itemize}
\end{theorem} 

   The second purpose of this paper is to study the relationship between parameter ideals of $M$ and parameter ideals of $M\otimes_RS$. We restrict our work to the assumption $\frak n=\frak mS$ in the hope of obtaining an elegant relationship between them. 
   
    Denote by  $X_M$ and $X_{M\otimes_RS}$ respectively the sets of all parameter ideals of $M$ and of $M\otimes_RS$.    
   Since $\frak n=\frak mS,$ it follows that the extension $\frak q S$ is a parameter ideal of $M\otimes_RS$ for every parameter ideal  $\frak q$ of $M$.    
   Therefore, we have a map $\varphi^e: X_M\to X_{M\otimes_RS}$ sending $\frak q$ to $\frak qS$   
    (we call $\varphi^e$ to be the {\it extension map} induced by $\varphi$). By the faithful flatness of $\varphi$, the extension map $\varphi^e$ is always injective, but it need not be a bijection, see Example \ref{E:1}. Note that  $\varphi^e$ is bijective if and only if each parameter ideal of  $M\otimes_RS$ is an extension of a parameter ideal of $M$. 
 
   It is natural to ask whether the contraction $\frak Q\cap R$ is a parameter ideal of $M$, provided that $\frak Q$ is a parameter ideal of  $M\otimes_RS$. In general, the answer is negative. We will construct in Example \ref{E:1} a flat local homomorphism  $(R,\frak m)\to (S, \frak n)$, where $R$ and $S$ are regular, $\frak n=\frak mS$  and there exists a parameter ideal $\frak Q$ of $S$ such that $\frak Q\cap R$ is not a  parameter ideal of $R$. This suggests studying the following set 
      $$X^*_{M\otimes_RS}=\{\frak Q\in X_{M\otimes_RS}\mid \frak Q\cap R\in X_M\}.$$ 
   Then we have a map $\varphi^c: X^*_{M\otimes_RS}\to  X_M$  sending $\frak Q$ to  $\frak Q\cap R$ (we call $\varphi^c$ to be the {\it contraction map} induced by $\varphi$). Since $\frak n=\frak mS$ and $\varphi$ is faithfully flat, it follows that $\frak qS\in X_{M\otimes_RS}$ and $\frak qS\cap R=\frak q$ for every $\frak q\in X_M$. Therefore, the contraction map $\varphi^c$ is always a surjection. However, $\varphi^c$ need not be bijective in general, see Example \ref{E:2}. Note that $\varphi^c$ is bijective if and only if each parameter ideal of  $M\otimes_RS$ belonging to  $X^*_{M\otimes_RS}$ is an extension of a parameter ideal of $M$.

   The following theorem, which is the second main result of this paper, provides criteria for the extension map $\varphi^e$ and the contraction map $\varphi^c$ to be bijections.

\begin{theorem} \label{T:2a} Assume that $\frak n=\frak mS.$ Let $M$ be a finitely generated $R$-module. 
\begin{itemize}
\item[{\rm (a)}] If $\ell_R(S/\frak mS)=1$ then the extension map $\varphi^e: X_M\to X_{M\otimes_RS}$ is a bijection.

\item[{\rm (b)}] If the local ring $R/\Ann_RM$ is Gorenstein or $\ell_R(S/\frak mS)=1$ then the contraction map $\varphi^c: X^*_{M\otimes_RS}\to X_M$ is a bijection.
\end{itemize}
\end{theorem}

In the next section, we prove the main results. We also provide remarks and examples to show that Theorem \ref{T:2a}(a) is optimal, whereas Theorem \ref{T:2a}(b) is not yet sharp.
 
\section{Main results}

  In order to prove Theorem \ref{T:1a}, we need some  preliminaries. 
  
  We recall some characterizations of Buchsbaum modules, see \cite[Lemma 1]{Vog}, \cite[Theorem 2.15, Theorem 20 (App.)]{SV}.  Let $x_1,\ldots, x_t$ be a minimal system of generators of the maximal ideal $\frak m$, let $K_{\bullet} (x_1,\ldots ,x_t; M)$ be the Koszul complex. Note that this complex is, up to isomorphism, independent of choice of the minimal system of generators of $\frak m$, so it is denoted by $K_{\bullet} (\frak m; M)$. For each $i$, denote by $H_i(\frak m; M)$ the $i$-th homology module of the Koszul complex $K_{\bullet} (\frak m; M)$ and  $H^i(\frak m; M)=H_{t-i}(\frak m; M).$  

\begin{lemma} \label{L:1a} Let $\dim_R(M)=d$. The following statements are equivalent:
\begin{itemize} 
\item[{\rm (a)}] $M$ is  Buchsbaum.
 \item[{\rm (b)}] The canonical maps $\rho^i_M: H^i(\frak m;M)\to H^i_{\frak m}(M)$ are surjective for all $i\neq d$.
\item[{\rm (c)}] There exists a minimal system of generators of $\frak m$ such that each its subsystem of $d$ elements is a standard s.o.p of $M$.
\end{itemize}
\end{lemma}

Note that the annihilators of finitely generated modules are compatible under
flat extensions. Although this elementary property is already known, we provide a proof for the reader's convenience.

\begin{lemma} \label{L:1b} $\Ann_S(M\otimes_RS)=(\Ann_RM)S.$
\end{lemma}
\begin{proof}
Let $m_1, \ldots , m_t$ be a system of generators of $M$. Then we have the exact sequence of $R$-modules
$$0\to \Ann_RM\to R\overset{\pi}{\to} M^t,$$ where  $\pi: R\to M^t$ sends $r$ to $(rm_1, \ldots , rm_t).$ This exact sequence induces the exact sequence
$$0\to (\Ann_RM)\otimes_RS\to S\overset{\pi^*}{\to} (M\otimes_RS)^t.$$
Hence $(\Ann_RM)S=(\Ann_RM)\otimes_RS=\Ker (\pi^*)=\Ann_S(M\otimes_RS).$
\end{proof}

\begin{theorem} \label{T:1} Suppose that $M$ is not Cohen-Macaulay. Then 
\begin{itemize}
\item[{\rm (a)}] $M\otimes_RS$ is  generalized Cohen-Macaulay if and only if  $M$ is   generalized Cohen-Macaulay and $\dim(S/\frak mS)=0$. In this case we have $I(M\otimes_RS)=\ell_S(S/\frak mS)~I(M).$
\item[{\rm (b)}] $M\otimes_RS$ is  Buchsbaum  if and only if $M$ is  Buchsbaum and $\frak n=\frak mS.$ In this case we have $I(M\otimes_RS)=I(M).$
\end{itemize}
\end{theorem}

 \begin{proof}  Let $\dim_R(M)=d$, $\depth_R(M)=s$ and $\dim(S/\frak mS)=r$. Then $s<d$ since $M$ is not Cohen-Macaulay.  Let $x_1, \ldots , x_s\in\frak m$ be an $M$-sequence. Set $N:=M/(x_1, \ldots , x_s)M.$ Then $\depth_R(N)=0$ and $\dim_R(N)=d-s>0$. Hence $H^0_{\frak m}(N)\neq 0$. Therefore,
$$\dim_S\big(H^0_{\frak m}(N)\otimes_RS\big)=\dim_R(H^0_{\frak m}(N))+\dim (S/\frak mS)=r<(d-s)+r.$$

(a) Assume that  $M\otimes_RS$ is generalized Cohen-Macaulay. As  $\dim_S(M\otimes_RS)=d+r$ and
$$\dim_S\big((M\otimes_RS)/(x_1, \ldots , x_s)(M\otimes_RS)\big)=\dim_S(N\otimes_RS)=d-s+r,$$ it follows that $x_1, \ldots , x_s$ is a part of a s.o.p of $M\otimes_R S$. Hence $N\otimes_RS$ is generalized Cohen-Macaulay of positive dimension $d+r-s$. Therefore,  $H^0_{\frak n}(N\otimes_RS)$ is the largest submodule of $N\otimes_RS$ of dimension less than $d+r-s.$ So, $H^0_{\frak m}(N)\otimes_RS \subseteq H^0_{\frak n}(N\otimes_RS).$  Therefore,  
$$r=\dim_S\big(H^0_{\frak m}(N)\otimes_RS\big)\leq \dim_S(H^0_{\frak n}(N\otimes_RS))=0.$$  Hence $\dim (M\otimes_RS)=d$ and $H^i_{\frak m}(M)\otimes_RS\cong H^i_{\frak n}(M\otimes_RS)$ by flat base change \cite[4.3.2]{BS}  for every integer $i\geq 0$. So we have $$\ell_R(H^i_{\frak m}(M))=\frac{\ell_S(H^i_{\frak n}(M\otimes_RS))}{\ell_S(S/\frak mS)}<\infty$$ for all $i<d.$ Therefore,  $M$ is  generalized Cohen-Macaulay.

Conversely, assume that $M$ is  generalized Cohen-Macaulay and $\dim (S/\frak mS)=0.$ Then $\dim_S(M\otimes_RS)=d$ and $H^i_{\frak m}(M)\otimes_RS\cong H^i_{\frak n}(M\otimes_RS)$ for all $i$. Hence $$\ell_S(H^i_{\frak n}(M\otimes_RS))=\ell_R(H^i_{\frak m}(M))~\ell_S(S/\frak mS)<\infty$$ for all $i<d.$ Therefore, $M\otimes_RS$ is generalized Cohen-Macaulay. 

Finally, assume that $M\otimes_RS$ is generalized Cohen-Macaulay. Then $I(M\otimes_RS)<\infty$ and we have by \cite[Satz 3.7]{CST} that 
\begin{align} I(M\otimes_RS)&=\sum_{i=0}^{d-1}\binom{d-1}{i}\ell_S(H^i_{\frak n}(M\otimes_RS))\notag\\
&=\sum_{i=0}^{d-1}\binom{d-1}{i}\ell_R(H^i_{\frak m}(M))~\ell_S(S/\frak mS)=\ell_S(S/\frak mS)~I(M).\notag\end{align}

 (b) Assume that $M\otimes_RS$ is  Buchsbaum. Then $r=0$ by the assertion (a). Therefore, $\dim_S(M\otimes_RS)=d$. Let $\frak q$ be a  parameter ideal of $M$. Set $\frak Q=\frak qS$. Since $r=0$, it follows that  $\frak Q$ is a parameter ideal of $M\otimes_RS$ and  
$$\ell_R(M/\frak qM)-e(\frak q; M)=\frac {\ell_S\big((M\otimes_RS)/\frak Q(M\otimes_RS)\big)-e(\frak Q; M\otimes_RS)}{\ell_S(S/\frak mS)}.$$
Therefore, $\ell_R(M/\frak qM)-e(\frak q; M)$ is a constant not depending on $\frak q$. Thus, $M$ is Buchsbaum. Now we prove that $\frak n=\frak mS.$ Since $M$ is not Cohen-Macaulay, there exists  $i<d$ such that $H^i_{\frak m}(M)\neq 0.$ Because $r=0$,  we have $$H^i_{\frak n}(M\otimes_RS)\cong H^i_{\frak m}(M)\otimes_RS\neq 0.$$ So,  $\Ann_RH^i_{\frak m}(M)\subseteq \frak m$ and $\Ann_SH^i_{\frak n}(M\otimes_RS)\subseteq \frak n$. Because $M$ is  a Buchsbaum  $R$-module and  $M\otimes_RS$ is a Buchsbaum $S$-module, we have $\frak mH^i_{\frak m}(M)=0$ and $\frak n H^i_{\frak n}(M\otimes_RS)=0,$ see \cite{SV}. Hence
 $\Ann_RH^i_{\frak m}(M)=\frak m$ and $\Ann_SH^i_{\frak n}(M\otimes_RS)=\frak n.$  Since $H^i_{\frak m}(M)$ is a finitely generated $R$-module,  we have by Lemma \ref{L:1b} that $\Ann_S(H^i_{\frak m}(M)\otimes_RS)=\frak mS$. Therefore,  $\frak n= \frak mS.$

Conversely, assume that $M$ is Buchsbaum and $\frak n=\frak mS.$ Then we get by Lemma \ref{L:1a} that  the canonical maps $\rho^i_M: H^i(\frak m;M)\to H^i_{\frak m}(M)$ are surjective for all $i<d$. Since $R\to S$ is faithfully flat and $\frak n=\frak mS$, it follows that $\rho^i_{M\otimes_RS}: H^i(\frak n;M\otimes_RS)\to H^i_{\frak n}(M\otimes_RS)$ are surjective for all $i<d$. Note that $\dim_S(M\otimes_RS)=d$.  Therefore, $M\otimes_RS$  is a Buchsbaum $S$-module by Lemma \ref{L:1a}. As $\frak n=\frak mS$, we have  $I(M\otimes_RS)=I(M)$ by the assertion (a).
\end{proof}

In the following part, we study the relationship between parameter ideals of $M$ and parameter ideals of $M\otimes_RS$ under the assumption  that $\frak n=\frak mS$. Note that the natural maps from $R$ to its $\frak m$-adic completion $\widehat R$ and to its Henselization $R^h$ satisfy this assumption. If $K\subseteq K'$ is a field extension and $R, S$ are  the formal power series rings in the same variables over $K, K'$ respectively, then the canonical injection $R\to S$  satisfies this assumption.

We first give an example to show that the contraction of a  parameter ideal of $M\otimes_RS$ is not necessarily a parameter ideal of $M$, even in the case where  $R$  and $S$ are regular local rings and $M=R$. It follows that the extension map $\varphi^e$ need not be  a bijection.  

\begin{example} \label{E:1} {\rm Let $\Bbb Q\subset K\subseteq \Bbb R$ be field extensions, where $K\neq \Bbb Q$. Let $R=\Bbb Q[[x,y]]$ and $S=K[[x,y]]$ be the formal power series rings in two variables over $\Bbb Q$ and $K$, respectively.  For $\alpha\in K\setminus\Bbb Q,$ set $\frak Q_{\alpha}=(x+\alpha y, x^2+y^2)$. Then $R, S$ are regular local rings with the maximal ideals $\frak m=(x,y)R$ and $\frak n=(x,y)S$, respectively. The embedding map $\varphi: R\to S$ is a local flat homomorphism with $\frak n=\frak mS$, $\frak Q_{\alpha}$ is a parameter ideal of $S$, but $\frak Q_{\alpha}\cap R$  is not a parameter ideal of $R$. In particular, the extension map $\varphi^e: X_R\to X_S$ is not a surjection}.
\end{example}

\begin{proof} Let $\alpha\in K\setminus\Bbb Q.$  We first claim that $\frak Q_{\alpha}\supseteq (x^2, y^2, xy)S.$ We have 
	$$xy-\alpha x^2 = xy +\alpha y^2 - \alpha x^2-\alpha y^2 =y(x+\alpha y) - \alpha (x^2+y^2)\in \frak Q_{\alpha}.$$ Therefore, we have
	$$x^2+\alpha^2x^2=x^2+\alpha xy - \alpha xy +\alpha^2x^2=x(x+\alpha y) - \alpha (xy-\alpha x^2)\in \frak Q_{\alpha}.$$
	Hence $x^2 (1+\alpha^2)\in \frak Q_{\alpha}.$ Note that $1+\alpha^2\neq 0$, so it is a unit in $S$. Hence $x^2\in \frak Q_{\alpha}$. So, we have $y^2=(x^2+y^2)-x^2\in \frak Q_{\alpha}.$ Hence $xy=y(x+\alpha y)-\alpha y^2\in \frak Q_{\alpha}.$ Therefore, $\frak Q_{\alpha}\supseteq (x^2, y^2, xy)S,$ the claim is proved.
	
	Next, we show that $\frak Q_{\alpha}\cap R=(x^2, y^2, xy).$ Let $\frak q=(x^2, y^2, xy)$. We get by the above claim that $\frak q\subseteq \frak Q_{\alpha}\cap R.$ Let $f\in\frak Q_{\alpha}\cap R.$ Since $\frak Q_{\alpha}\cap R\subseteq (x, y)R,$ we can write $$f=x^2f_1+y^2f_2+xyf_3+f_4,$$
	 where $f_1, f_2, f_3, f_4\in R$ and $f_4=ax+by$ for some $a, b\in\Bbb Q.$ Since $\frak qS\subseteq \frak Q_{\alpha}$ and $f\in\frak Q_{\alpha}$, it follows that $f_4\in\frak Q_{\alpha}=(x+\alpha y, x^2+y^2).$ Therefore, $f_4=\beta (x+\alpha y)$ for some $\beta\in K.$ Hence $\beta=a\in\Bbb Q$ and $\beta\alpha=b\in\Bbb Q.$ If $\beta\neq 0$ then $\alpha =b/a\in\Bbb Q,$ this is impossible. Hence $\beta =0$ and hence $f_4=0$. Therefore, $f=x^2f_1+y^2f_2+xyf_3\in\frak q.$ Thus, $\frak Q_{\alpha}\cap R=\frak q.$ It is clear that $\frak q$ is not a parameter ideal of $R$.
	
	Note that $K[[x,y]]/(x^2, y^2, xy)$ forms a $K$-vector space with a basis $\{1, x, y\},$ therefore $\ell_S(S/\frak q S)=3.$ Since $\frak qS\subseteq \frak Q_{\alpha}$, we can check that  $\frak Q_{\alpha}=(x^2, y^2, xy, x+\alpha y).$ Hence $S/\frak Q_{\alpha}$ forms a $K$-vector space with a basis $\{1, y\}$ and hence   $\ell_S(S/\frak Q_{\alpha})=2.$ Therefore, $\frak Q_{\alpha}\neq \frak qS.$ Suppose to contrary that the extension map $\varphi^e: X_R\to X_S$ is a bijection. Then there exists a parameter ideal $\frak q'$ of $R$ such that $\frak Q_{\alpha}=\frak q'S.$ Hence $\frak q'=\frak Q_{\alpha}\cap R=\frak q$ and hence $\frak Q_{\alpha}=\frak qS,$ a contradiction. Therefore, $\varphi^e$ is not  bijective.
\end{proof}

Next, we show that there exist distinct parameter ideals $\frak Q, \frak Q'$ of $S$ and a parameter ideal $\frak q$ of $R$ such that $\frak Q\cap R=\frak Q'\cap R=\frak q$. So, the contraction map $\varphi^c$ need not be bijective.    

\begin{example} \label{E:2} {\rm Let $K=\Bbb Q[\sqrt{2}]$ and  $R=\Bbb Q[[x,y]]/(xy, y^2)$, $S=K[[x,y]]/(xy, y^2),$ where $\Bbb Q[[x,y]]$ and $K[[x,y]]$ are the formal power series rings in two variables $x, y$  over $\Bbb Q$ and $K$, respectively. Let $\frak m=(x,y)R$, $\frak n=(x,y)S$,  $\frak Q=(x-\sqrt{2}y)S$ and $\frak q=x^2R$. Then $\frak q$ is a parameter ideal of $R$, the embedding map $\varphi: R\to S$ is a flat local homomorphism with $\frak n=\frak mS$,  and $\frak Q, \frak qS$ are distinct parameter ideals of $S$ satisfying $\frak Q\cap R=\frak qS\cap R=\frak q.$  In particular, the contraction map $\varphi^c: X^*_S\to X_R$ is not injective}.
\end{example}

\begin{proof}  Since $x^2-x(x-\sqrt{2}y)\in (xy, y^2),$ we have
	$$\frak qS=x^2S=x(x-\sqrt{2}y)S=x\frak Q.$$ Hence $\frak qS\neq \frak Q$ by Nakayama Lemma and $\frak q\subseteq \frak Q\cap R$.  Suppose to contrary that $\frak q\neq \frak Q\cap R$. Then $\frak (Q\cap R)/\frak q$ is a non-zero ideal of $\Bbb Q[[x,y]]/(x^2, xy, y^2).$ Therefore, there exist $a, b\in \Bbb Q$ and $\beta\in K$ such that $0\neq ax+by=\beta (x-\sqrt{2}y).$ If $a\neq 0$ then $\sqrt{2}=-b/a,$ it is impossible. Therefore, $\beta =a=0$, hence $ax+by=0,$ a contradiction. Thus, $\frak Q$, $\frak qS$ are  distinct parameter ideals of $S$ and $\frak q$ is a parameter ideal of $R$ satisfing $\frak Q\cap R=\frak qS\cap R=\frak q.$
\end{proof}

 In order to prove Theorem \ref{T:2a}, we need recall some  preliminaries. It is well known that if  $N$ is a submodule  of $M$, then $N$  can be expressed as an irredundant intersection of irreducible submodules of $M$, and the number of irreducible submodules appearing in such an expression is independent of the choice of the expression (this result was proved by E. Noether  \cite{Noe}  for ideals of Noetherian rings). This invariant is called the {\it reducibility index} of $N$  in $M$ and denoted by ${\rm ir}_M(N)$.   If $\frak q$ is a parameter ideal of $M$, then ${\rm ir}_M(\frak q M)$  is called the {\it reducibility index} of $\frak{q}$ on $M$. We recall some important results about reducibility index.  

For an $R$-module $L$, we set $\Soc (L)=(0:_L\frak m).$ Note that if $L$ is either finitely generated or Artinian then $\dim_{R/\frak m}\Soc (L)<\infty.$ 
 
\begin{remark} \label{R:2} {\rm Let $\dim_R(M)=d$. Let $\frak q$ be a parameter ideal of $M$}.
\begin{itemize}
\item[{\rm (a)}] {\rm $\ir_M(\frak qM)=\dim_{R/\frak m}\Soc (M/\frak qM),$ see \cite{Nor} (see also \cite[Lemma 2.3]{CQT})}.  

\item[{\rm (b)}] {\rm $R$ is Gorenstein if and only if $\ir_R(\frak q)=1$ for every parameter ideal $\frak q$ of $R$, see \cite {NR}.

\item[{\rm (c)}] {\rm  If $M$ is Cohen-Macaulay then $\ir_M(\frak qM)$ is a constant  independent of $\frak q$, see  \cite{Nor}.} 

\item[{\rm (d)}] {\rm If $M$ is generalized Cohen-Macaulay, then  we have by \cite{GSu} that} $${\rm ir}_M(\frak q M) \leq \sum_{j=0}^{d-1}{\it{{{d} \choose {j}}}}\ell_R(H^j_{\frak m}(M))+ \dim_{R/\frak m} \Soc (H^d_{\frak m}(M)).$$ 
Moreover, by \cite{CQ}, if $\frak q$ is contained in a large enough power of $\frak m$ then}
$${\rm ir}_M(\frak q M) = \sum_{j=0}^{d}{\it{{{d} \choose {j}}}}\dim_{R/\frak m} \Soc(H^j_{\frak m}(M)).$$
\item[{\rm (e)}] {\rm If $\dim_RH^i_{\frak m}(M)\leq 1$ for all $i<d$ then there exists $C\in\Bbb N$ such that ${\rm ir}_M(\frak q M) < C$ for all parameter ideals $\frak q$ of $M$, see \cite{Q}.}  
\end{itemize}
 \end{remark} 

Now we are ready to prove Theorem \ref{T:2a}.
 
\begin{theorem} \label{T:2} Assume that $\frak n=\frak mS.$ Let $M$ be a finitely generated $R$-module.
\begin{itemize}
\item[{(a)}] If $\ell_R(S/\frak mS)=1$ then every parameter ideal $\frak Q$ of $M\otimes_RS$ is an extension of a  parameter ideal $\frak q$ of $M$. In this case, the extension map $\varphi^e: X_M\to X_{M\otimes_RS}$ is  bijective.
		
\item[{(b)}] If $R/\Ann_RM$ is Gorenstein or $\ell_R(S/\frak mS)=1$, then every parameter ideal $\frak Q$  of $M\otimes_RS$ belonging to $X^*_{M\otimes_RS}$ is an extension of a  parameter ideal $\frak q$ of $M$. In this case, the contraction map $\varphi^c: X^*_{M\otimes_RS}\to X_M $ is bijective.
	\end{itemize}
\end{theorem}

\begin{proof} (a) We first note that  the induced map $R/\Ann_RM\to S/(\Ann_RM)S$ is also a flat local homomorphism. By Lemma \ref{L:1b}, without loss of any generality, we may assume that $\Ann_RM=0$ and $\Ann_S(M\otimes_RS)=0.$ 

Now we prove the statement (a). Let $\frak Q$ be a parameter ideal $M\otimes_RS$. Then $\frak Q$ is $\frak n$-primary as $\Ann_S(M\otimes_RS)=0.$  So, there exists $k\in \Bbb N$ such that $\frak m^kS=\frak n^k\subseteq \frak Q$.  Consider the map $\varphi_{k+1}: R/\frak m^{k+1}\rightarrow S/\frak m^{k+1}S$ given by $\varphi_{k+1}(r+\frak m^{k+1})=\varphi(r)+\frak m^{k+1}S.$  It follows by the faithful flatness of $\varphi$ that $\varphi_{k+1}$ is an injective homomorphism of $R$-modules. Since $\ell_R(S/\frak n)=1$, we have
$\ell_R(S/\frak m^{k+1}S)=\ell_S(S/\frak m^{k+1}S)~\ell_R(S/\frak n)=\ell_S(S/\frak m^{k+1}S)$. Moreover,  $\ell_S(S/\frak mS)=1$ as $\frak n=\frak mS$. So, by the faithful flatness of $\varphi$ we have 
$$\ell_R(S/\frak m^{k+1}S)=\ell_S(R/\frak m^{k+1}\otimes_RS)=\ell_R(R/\frak m^{k+1})~\ell_S(S/\frak mS)=\ell_R(R/\frak m^{k+1})=\ell_R(\im (\varphi_{k+1})).$$ 
 Hence,  $\varphi_{k+1}$ is an isomorphism of $R$-modules. Let $\dim_R(M)=d.$ Then $\dim_S(M\otimes_RS)=d.$ Suppose that $\frak Q=(y_1, \ldots , y_d).$  For each $i\leq d,$ since $\varphi_{k+1}$ is an isomorphism,  there exists $x_i\in R$ such that $\varphi(x_i)+\frak m^{k+1}S=y_i+\frak m^{k+1}S.$ Hence $\varphi(x_i)-y_i\in \frak m^{k+1}S\subseteq \frak m\frak Q.$ Hence $\varphi(x_i)\in\frak Q$ for all $i\leq d$, and hence $(x_1, \ldots , x_d)S \subseteq \frak Q.$ It follows that  $$\frak Q=(y_1, \ldots , y_d)=(x_1, \ldots , x_d)S+\frak m\frak Q.$$ 
So, $\frak Q= (x_1, \ldots , x_d)S$ by Nakayama Lemma. Set $\frak q:=\frak Q\cap R.$ Then $\frak q=(x_1, \ldots , x_d)$ and $\frak Q=\frak qS$. Thus,   $\frak Q$ is the extension of the parameter ideal $\frak q$ of $M$. In particular, the extension map $\varphi^e: X_M\to X_{M\otimes_RS}$ is a bijection.

(b) Suppose that $R/\Ann_RM$ is Gorenstein. By Lemma \ref{L:1b}, without loss of generality, we may assume that $\Ann_RM=0$ and $R$ is Gorenstein. Let $\frak Q$ be a parameter ideal $M\otimes_RS$ such that $\frak Q\cap R$ is a parameter ideal of $M$. Set $\frak q=\frak Q\cap R.$ Note that $\frak q$ is a parameter ideal of $R$ since $\Ann_RM=0$. We claim that $\frak Q=\frak qS.$ In fact, it is clear that $\frak qS\subseteq \frak Q.$ Suppose to the contrary that $\frak qS\neq \frak Q.$  Since $R$ is Gorenstein and $\frak n=\frak mS$, it follows  by Remark \ref{R:2}(a),(b) that 
$$\ell_S\big((\frak qS:_S\frak mS)/\frak qS\big)=\ell_R\big((\frak q:\frak m)/\frak q\big)~\ell_S(S/\frak mS)=\dim_{R/\frak m}(\Soc (R/\frak q))=1.$$
So,  $(\frak qS:_S\frak mS)/\frak qS$ forms an $S/\frak n$-vector space of dimension $1$. Since $\ell_S(\frak Q/\frak qS)<\infty$, there exists an ideal $\frak I$ of $S$ such that $\frak qS\subset \frak I\subseteq \frak Q$ and  $\frak I/\frak qS$ forms an $S/\frak n$-vector space of dimension $1$. Let $s\in \frak I\setminus \frak qS$. Then $\frak n s\in\frak qS$ and $(sS+\frak qS)/\frak qS=\frak I/\frak qS$. Hence $s\in (\frak qS:_S\frak n)=(\frak qS:_S\frak mS)$ and hence $(sS+\frak qS)/\frak qS=(\frak qS:_S\frak mS)/\frak qS.$ So, we have $$(\frak q:_R\frak m)S=(\frak qS:_S\frak mS)=\frak I\subseteq \frak Q.$$ Therefore, 
$(\frak q:_R\frak m)\subseteq \frak Q\cap R=\frak q$, this gives a contradiction. So, the claim is proved. Let $\frak Q, \frak Q'\in X^*_{M\otimes_RS}$ such that $\frak Q\cap R=\frak Q'\cap R=\frak q.$  Then $\frak Q=\frak qS=\frak Q'$ by the above claim. Thus, the contration map $\varphi^c$ is a bjection.

Next, suppose that $\ell_R(S/\frak mS)=1$. Note that the contration map $\varphi^c$ is always surjective. Let $\frak Q, \frak Q'\in X_{M\otimes_RS}^*$ such that $\frak Q\cap R=\frak Q'\cap R.$ By assertion (a), there exist $\frak q, \frak q'\in X_M$ such that $\frak Q=\frak qS$ and $\frak Q'=\frak q'S$. Hence $\frak q=\frak Q\cap R=\frak Q'\cap R=\frak q'.$  Therefore, $\frak Q=\frak Q'$. Thus, $\varphi^c$ is injective.
\end{proof}

 \begin{remark} \label{R:3} {\rm  The result of Theorem \ref{T:2}(a) is optimal in the sense that if the length $\ell_R(S/\frak mS)$ is greater than 1, the extension map $\varphi^e$ may no longer be a bijection.  
 		
 For example, consider the embedding local flat homomorphism  $\varphi: R \to S$ constructed in Example \ref{E:1}, where $\Bbb Q\subset K\subseteq \Bbb R$ are field extensions with $K\neq \Bbb Q$, $R=\Bbb Q[[x,y]]$ and $S=K[[x,y]]$.    Let $\frak m=(x,y)R$ and $\frak n=(x,y)S$ denote  the maximal ideals of $R$ and $S$ respectively. Then $\frak n=\frak mS$. As shown in Example \ref{E:1}, the extension map $\varphi^e: X_R\to X_S$ is not bijective. 	
 If we choose $K$ to be a field extension of infinite degree over $\Bbb Q$ (for example, $K=\Bbb R$), then  $\ell_R(S/\frak mS)=\infty$.   For an integer $n>1,$ if we choose  $K=\Bbb Q[\alpha]$  to be a field extension of degree $n$ over $\Bbb Q$ (for example $\alpha=\sqrt[n]{2}$, the real $n$-th root of $2$), then  $\ell_R(S/\frak mS)=n$.}
\end{remark} 

The result of Theorem \ref{T:2}(b) is not sharp in some sense. Even $R$ is not Gorenstein and $\ell_R(S/\mathfrak{m}S)$ is arbitrarily large, the contraction map $\varphi^c: X^*_{M\otimes_RS}\to X_M$ can still be a bijection for every finitely generated $R$-module $M$.

 \begin{example} \label{E:3}  {\rm Let $t$ be an indeterminate, $R = \mathbb{Q}[[t^3, t^4, t^5]]$ and $S = K[[t^3, t^4, t^5]]$, where $\Bbb Q\subset K\subseteq \Bbb R$ are field extensions. Let $\frak m= (t^3, t^4, t^5)R$ and $\frak n=(t^3, t^4, t^5)S$ denote  the maximal ideals of $R$ and $S$, respectively. Then the embedding map $\varphi: R\to S$ is a flat local homomorphism, $\frak n=\frak mS$, the contraction map $\varphi^c: X^*_{M\otimes_RS}\to X_M$ is bijective for every finitely generated $R$-module $M$, but $R$ is not Gorenstein and $\ell_R(S/\frak mS)$ can be arbitrarily large.} 
 \end{example}
  
\begin{proof}  Let $\frak q=t^3R.$ Then $\frak q$ is a parameter ideal of $R$. We can check that $(0:_{R/\frak q}\frak m)$ is a vector space over $R/\frak m$ of dimension $2$ with a basis $\{T^4, T^5\},$ where $T^4, T^5$ are respectively the images of $t^4, t^5$ in $R/\frak q$. Hence $\ir_R(\frak q)=\dim_{R/\frak m} \Soc(R/\frak q)=2$, and hence $R$ is not Gorenstein by Remark \ref{R:2}(b).
 
 Let $M$ be a finitely generated $R$-module of dimension $d$. As $\dim (R)=1$, we have $d\leq 1.$ Note that the contraction map $\varphi^c: X^*_{M\otimes_RS}\to X_M$ is always surjective. If $d=0$ then there is nothing to do. So we can assume that $d=1$.  Let $\frak Q, \frak Q'$ be parameter ideals of $M\otimes_RS$ and let $\frak q$ be a parameter ideal of $M$ such that $\frak Q\cap R=\frak Q'\cap R=\frak q$. We will show that $\frak Q=\frak Q'=\frak qS.$ 
  Suppose to contrary that $\frak Q\neq \frak qS$. Since $d=1$ and $\dim_S(M\otimes_RS)=1$, we have $\frak Q=fS$ and $\frak q=gR$ for some non-zero elements  $f\in \frak n$ and $g \in \frak m.$ Hence   $fS \neq gS$. Note that $\frak qS\subseteq \frak Q$.  So, we can write $g = fh$ for some $h \in S$. Note that  $h$ is not a unit in $S$ since  $fS\neq gS$. It follows that $h \in \mathfrak n$. Note that $t^k\in R$ and $t^k\in S$ for all integers $k\geq 3.$ Therefore, $\frak m=t^3 \mathbb{Q}[[t]]$, $\frak n=t^3 K[[t]]$, $t^4 \mathbb{Q}[[t]] \subset R$ and $t^4 K[[t]] \subset S.$ Since $h\in\frak n$, we have $th\in t^4 K[[t]]$. So, $th \in S$. Similarly, since $g \in \mathfrak{m},$ we have  $tg \in t^4 \mathbb{Q}[[t]]$, and hence $tg\in R$.  Since $g=fh$ and $th\in S$, we have $tg=f(th)\in fS.$ Since $tg\in R$, we have $tg\in fS\cap R=\frak q=gR.$ Therefore, $tg=gr$ for some $r\in R.$ 
 Since $K[[t]]$ is a domain and $g \neq 0$, we can apply the cancellation law in $K[[t]]$ for $g$, yielding $t=r\in R=\mathbb{Q}[[t^3, t^4, t^5]]$. This gives a contradiction. Hence $\frak Q=\frak qS$. By the same arguments we have $\frak Q'= \frak qS$.  Hence the contraction map $\varphi^c$ is injective.
 
 Finaly, as shown in Remark \ref{R:2}, if we choose $K=\Bbb R$ then  $\ell_R(S/\frak mS)=\infty$; if we choose $K=\Bbb Q[\sqrt[n]{2}]$ then $\ell_R(S/\frak mS)=n$ for any given integer  $n>1.$ 
\end{proof}

In Example \ref{E:2}, the local ring $R=\Bbb Q[[x,y]]/(xy, y^2)$ is not Cohen-Macaulay, and the contraction map $\varphi^c$ is not a bijection. Meanwhile, in Example \ref{E:3}, although the local ring $R=\Bbb Q[[t^3, t^4, t^5]]$ is not Gorenstein, it is Cohen-Macaulay and the contraction map $\varphi^c$ is a bijection. Based on Theorem \ref{T:2}(b) and these observations, we conclude this paper with the following open question. 

\begin{question} {\rm Let $\varphi: (R,\frak m)\to (S,\frak n)$ be a local flat homomorphism such that $\frak n=\frak mS$. Let $M$ be a finitely generated $R$-module.  Suppose that $R/\Ann_RM$ is Cohen-Macaulay. Is the contraction map $\varphi^c: X^*_{M\otimes_RS}\to X_M$  bijective? }	
\end{question}

\noindent {\bf Acknowledgment}

The authors thank  Nguyen Dang Hop for the suggestions concerning  Example \ref{E:2}.

\end{document}